\documentclass{amsart}
\usepackage{amsmath}
\usepackage{amsfonts}
\usepackage{amssymb}
\usepackage{amsthm}
\usepackage{tikz}
\usepackage{graphicx}
\usetikzlibrary{matrix,arrows,decorations.pathmorphing}

\newtheorem{lemma}{Lemma}
\newtheorem{definition}{Definition}

\author{Patrik Vaclav Nabelek}
\email{pnabelek@math.arizona.edu}
\author{Douglas Pickrell}
\email{pickrell@math.arizona.edu}
\title{Harmonic Maps and the Symplectic Category}

\begin{document}

\maketitle

\begin{abstract} In the context of the two dimensional sigma model, we show that classical field theory
naturally defines a functor from Segal's category of Riemann surfaces to the Guillemin-Sternberg/Weinstein
category of canonical relations in symplectic geometry, following ideas of Cattaneo, Mnev, and Reshetikhin.
This is an expository article, based on a research tutorial.
\end{abstract}

\section{Introduction}

Given a smooth map of manifolds $f:M\to N$, there is an associated map of vector bundles, the derivative $f_*:TM \to TN$. The chain rule asserts that if $f : M\to N$ and $g : N \to P$, then $(g\circ f)_*  = g_* \circ f_*$.
In the language of category theory, the derivative defines a functor from the category of manifolds to the category of vector bundles.

In Hamiltonian mechanics, the dynamics is given on the cotangent bundle of a configuration space (or more generally on a symplectic or Poisson manifold), rather than on the tangent bundle. However there is not a functor that takes a smooth map $f:M \to N$ to a map from $T^*M$ to $T^*N$.

In \cite{GuSt1} and \cite{Wein} (and more recently, and in more detail, in \cite{GuSt2}) a ``symplectic category" $\textbf{Symp}$ is defined, in which the objects are symplectic manifolds and the morphisms are Lagrangian submanifolds, which are also called canonical relations (this is not quite a true category, but we will temporarily ignore this point). Given a smooth map $f:M \to N$, Guillemin and Sternberg, and Weinstein, associate a Lagrangian submanifold of $T^*M^- \times T^*N$ to the map $f$, given in an elegant way by the conormal bundle of the graph of $f$. Their beautiful discovery is that this assignment, analogous to the derivative, defines a functor from the category of manifolds to the symplectic category.

This point of view appears to be very fruitful in field theory. At the classical and semiclassical levels, in the context of gauge theory, it has been developed by Cattaneo, Mnev, and Reshetikhin (\cite{Cat1}, \cite{Cat2}, \cite{R}). In this paper, which is expository, we consider the two dimensional nonlinear sigma model. In this case the classical fields of the model are maps from a Riemann surface to a fixed target Riemannian manifold, $N$, and the classical solutions are harmonic maps. The case when $N=U/K$, a compact symmetric space, is especially interesting, because in some senses the model is integrable, both classically and - for
some, but apparently not all, symmetric space targets - quantum mechanically. In addition, on the one hand, the two dimensional sigma model has a number of (partially elucidated) characteristics in common with four dimensional gauge theory (essentially the standard model), such as a classical conformal symmetry, which is broken at the quantum level, and asymptotic freedom, which can be interpreted to mean that the semiclassical theory should be a good approximation to the quantum theory at short distances (see \cite{Witten}, for an expository account of this); and, on the other hand, the sigma model is more elementary than gauge theory (hence presumably more tractable).

The categorical point of view was introduced in field theory by Segal (see \cite{Segal}). The relevant category for the two dimensional sigma model - at the classical level - is Segal's category of compact Riemann surfaces, where the objects are compact oriented 1-manifolds, and the morphisms are compact Riemann surfaces. Our primary goal in this paper is to show that the classical theory naturally yields a functor from Segal's category to the symplectic category. More precisely given a compact oriented 1-manifold $S$, we associate to $S$ the cotangent bundle of the configuration space $\Omega^0(S;N)$ (essentially closed strings in $N$ parameterized by $S$), and to a compact Riemann surface $\Sigma$, we associate $Harm(\Sigma; N)$, the space of harmonic maps from $\Sigma$ to $N$, which in a natural way defines a Lagrangian submanifold of $T^*\Omega^0(S;N)$ (or more precisely, the space of Cauchy data along the boundary $S$; see
Section \ref{functor} below).

\begin{lemma}\label{sewing}Consider a composition $\Sigma_2\circ \Sigma_1$ in Segal's category. Then
$$Harm(\Sigma_2; N) \circ Harm(\Sigma_1; N) = Harm(\Sigma_2\circ\Sigma_1; N)$$\end{lemma}

On the left hand side of the equation in the lemma, $\circ$ denotes composition in the infinite dimensional generalization of $\textbf{Symp},$ while on the right hand side $\circ$ denotes composition in Segal's category of Riemann surfaces.

At one level this lemma is transparent: it is in particular asserting that a harmonic map on the composition
$\Sigma_2\circ\Sigma_1$ is the same thing as a pair of harmonic maps on $\Sigma_1$ and $\Sigma_2$ which agree
along the boundary in an appropriate sense (to first order); this is very reasonable, given the local nature of the harmonic map equation. Segal emphasized the importance of this sewing property for classical solutions in the construction of the corresponding quantum conformal field theory, in the one case where this is understood, i.e. when $N$ is a flat torus (see chapter 10 of \cite{Segal}). Our modest goal is simply to make explicit this connection between Segal's approach to conformal field theory, in a classical setting, and the symplectic category.

As we mentioned previously, this paper is expository - there is no claim to originality; it is based on a research tutorial. In proving the functor lemma we will, for the most part, follow the program in \cite{Cat1}.

\section{The Symplectic ``Category"}

In this section, following \cite{GuSt2}, we recall the definition of the symplectic category $\textbf{Symp}$.
Before we begin to define the symplectic category in detail, we need to define an involution and a product on the set of symplectic manifolds.
If $(M,\omega)$ is a symplectic manifold we will write $M$ for $(M,\omega)$ and $M^-$ for $(M,-\omega)$.
Note that $M\to M^-$ is clearly an involution.
If $(M_1,\omega_1)$ and $(M_2,\omega_2)$ are symplectic manifolds, then we define the product $M_1 \times M_2$ to be the symplectic manifold $(M_1 \times M_2, \pi_1^*\omega_1+\pi_2^*\omega_2),$ where the $\pi_i:M_1 \times M_2 \to M_i$ are the projections.

Given two symplectic manifolds $M_1$ and $M_2$, a canonical relation $M_1 \stackrel{L}{\rightarrow} M_2$ is a Lagrangian submanifold $L\subset M_1^- \times M_2.$
The terminology canonical relation is used, since the graph of a canonical transformation of phase space is a canonical relation, and of course subsets of $M_1^- \times M_2$ are relations.
A point in $\textbf{Symp}$ is identified with a Lagrangian submanifold $L$ of a symplectic manifold $M.$
This notion of a point is the most natural notion in the sense that a zero dimensional manifold is trivially isomorphic to $(\mathbb{R}^0,0),$ and thus a canonical relation $\mathbb{R}^0 \to M$ is a Lagrangian submanifold of $M \equiv (\mathbb{R}^0\times M, -0+\omega).$

Given a symplectic manifold $M$ we define the identity canonical relation $M \xrightarrow{\Delta_M} M$ to be $\{(p,p) \in M^- \times M: p\in M\}.$
It is clear that $\Delta_M \equiv M$ as a smooth manifold and that the symplectic form $- \pi_1^*\omega+\pi_2^*\omega $ vanishes on $\Delta_M,$; hence $\Delta_M$ is Lagrangian in $M^-\times M$ since it has the correct dimension.
Let $M_1 \xrightarrow{L_1} M_2$ and $M_2 \xrightarrow{L_2} M_3$ be two canonical relations.
Consider the subset $$L_2\star L_1 \subset L_1\times L_2 \subset M_1^-\times M_2\times M_2^-\times M_3$$ defined as
$$L_2 \star L_1 = \{(x_1,y_1,y_2,z_2) \in L_1 \times L_2 : y_1 = y_2 \}.$$
Let $\pi_{13}$ denote the projection of $M_1^-\times M_2\times M_2^-\times M_3$ onto $M_1^- \times M_3.$
\begin{definition}[Composition of Canonical Relations]
We say $L_1$ and $L_2$ are composable canonical relations when $\pi_{13}(L_2 \star L_1)$ is a Lagrangian submanifold of $M_1^- \times M_3.$
When $L_1$ and $L_2$ are composable we define the composition $L_2 \circ L_1 = \pi_{13}(L_2 \star L_1).$
\end{definition}

Let $\phi :M_1\to M_2$ be a local symplectomorphism of symplectic manifolds, i.e. $\omega_1=\phi^* \omega_2$ and $\phi$ is a regular smooth map.
Consider $L_\phi \subset M_1^-\times M_2$ given by $\{(p_1,p_2):p_2=\phi(p_1)\}.$
Let $(p_1,p_2)\in L_\phi,$ and $c$ a smooth parameterized curve in $L_\phi$ with $c(0)=(p_1,p_2).$
By the definition of $L_\phi,$ $c = (\tilde{c}, \phi\circ\tilde{c})$ where $\tilde{c}$ is a smooth parameterized curve in $M_1.$
When we consider vectors in $T_{(p_1,p_2)}L_\phi $ as equivalence classes of curves, we thus get that $v\in T_{(p_1,p_2)}L_\phi$ is of the form $(\tilde{v},\phi_*\tilde{v})$ where $\tilde{v} \in T_{p_1}M_1.$
Therefore,
\begin{align*}(-\pi_1^*\omega_1+\pi_2^*\omega_2)(v,w) & = -\omega_1(\tilde{v},\tilde{w}) + \omega_2(\phi_*\tilde{v},\phi_*\tilde{w}) \\ & =-\omega_1(\tilde{v},\tilde{w})+\phi^*\omega_2(\tilde{v},\tilde{w})=0.\end{align*}
Therefore $L_\phi$ is a Lagrangian submanifold of $M_1^-\times M_2.$ Therefore, to each symplectomorphism $\phi:M_1\to M_2$, we get a canonical relation $L_\phi\subset M_1^-\times M_2$, which we view as a morphism $L_{\phi}:M_1\to M_2$ in $\textbf{Symp}.$

We will now show that canonical relations of the form above, that come from local symplectomorphisms, are composable.
Since the Hamiltonian flow on $T^*M$ is a 1-parameter group of symplectomorphisms, this will show that the Hamiltonian flow corresponds to a functor from $\mathbb{R}$ (or at least $(-\epsilon,\epsilon)$) to $\textbf{Symp}.$
Let $\phi_1:M_1\to M_2$ and $\phi_2:M_2\to M_3$ be local symplectomorphisms. By definition
$$L_{\phi_2} \star L_{\phi_1} = \{(p_1, \phi_1(p_1), p_2,\phi_2(p_2)) : \phi_1(p_1) = p_2,$$
If $\phi_1(p_1) = p_2$, then $\phi_2\circ\phi_1(p_1) =\phi(p_2).$
Therefore $$L_{\phi_2} \circ L_{\phi_1} = \pi_{13}(L_{\phi_2} \star L_{\phi_1})=\{(p_1,\phi_2\circ\phi_1(p_1)):p_1\in M_1\}=L_{\phi_2\circ\phi_1}.$$
But we already know that $L_{\phi_2\circ\phi_1}$ is a Lagrangian submanifold, so $L_{\phi_2}$ and $L_{\phi_1}$ are composable and $L_{\phi_2}\circ L_{\phi_1}=L_{\phi_2\circ\phi_1}.$

This shows that the category $\textbf{Symp}$ is a natural extension of the category consisting of symplectic manifolds
and local symplectomorphisms. We now want to show that there is a functor from the category of manifolds to
 $\textbf{Symp}$, as we alluded to in the introduction.

To understand the definition of the functor, we first need to recall the definition of the conormal bundle of a submanifold $U\subset M$ ($M$ is now simply a manifold, or a configuration space).
The conormal bundle of $U$ is a special Lagrangian submanifold of $T^*M$ with the canonical symplectic form.
The conormal bundle of $U$ is defined as
$$N(U) =  \{\xi|_x \in T^*M : x\in U, \xi|_x(v) = 0 \text{ for all } v\in T_x^*U\}.$$
Note that if $M$ has a Riemannian metric then $N(U)$ is isomorphic to the normal bundle by the isomorphism of the tangent bundle to the cotangent bundle given by the Riemannian metric induced from the Riemannian metric on $M.$
Let $\alpha$ be the action one form on $T^*M,$ i.e. for $v\in T(T^*M)|_{\xi|_x}$ $\alpha(v)=\xi(\pi_*v).$
Then for $v\in T(N(U))$ $\alpha(v) = 0$ since $\pi_*v \in TM,$ so the conormal bundle is in fact a Lagrangian submanifold of $T^*M.$

Let $M_1$ and $M_2$ be manifolds, and $f:M_1 \to M_2$ be a smooth map.
Then $T^*M_1$ and $T^*M_2$ are symplectic manifolds with canonical symplectic forms $\omega_1, \omega_2$ and action one forms $\alpha_1, \alpha_2.$
The map $f$ induces a canonical relation on the cotangent bundles $L_f \subset T^*M_1^- \times T^*M_2$ defined as
\begin{equation}\label{defn3}L_f = \{(f^*\xi|_x,\xi|_{f(x)})\in T^*M_1\times T^*M_2\}.\end{equation}

\begin{lemma} $L_f$ is a canonical relation.\end{lemma}

\begin{proof}
First note that $(-f^*\xi|_x,\xi_{f(x)})$ is an element of $N(\text{graph}(f)).$ This is because a vector $v\in T(\text{graph}(f))|_{(x,f(x))}\subset TM_1|_x\times TM_2|_{f(x)}$ is of the form $(\tilde v|_x,f_*\tilde v|_{f(x)})$  where $\tilde v|_x \in TM_1|_x,$ and
$$(-f^*\xi|_x,\xi|_{f(x)})((\tilde v|_x,f_*\tilde v|_{f(x)})) = -f^*\xi(\tilde v|_x) + \xi(f_* \tilde v|_{f(x)}) = (-f^*+f^*)(\tilde v|_x) = 0.$$
Therefore $L_f^-=\{(-f^*\xi|_x,\xi|_{f(x)})\in T^*M_1\times T^*M_2\}$ is an $n+m$ dimensional submanifold of the $2n+2m$ space $T^*M_1\times T^*M_2.$
$L_f^-$ is also a subspace of the Lagrangian manifold $N(\text{graph}(f)) \subset T^*M_1\times T^*M_2.$
The previous two statements imply that $L_f^-$ a Lagrangian  submanifold of $T^*M_1\times T^*M_2.$
Consider the operation
$$i_1 : T^*M_1\times T^*M_2 \to T^*M_1\times T^*M_2 : (\xi_1,\xi_2) \to (-\xi_1,\xi_2),$$
then $i_1(L_f^-) = L_f.$
Let $$v = (v_1|_{f^*\xi|_x},v_2|_{\xi|_{f(x)}}) \in T(T^*M_1\times T^*M_2)|_{(f^*\xi|_x,\xi|_{f(x)})},$$
then
$$i_1^*(\alpha)(v) = i_1(f^*\xi|_x, \xi_{f(x)})(d\pi i_{1*}(v)) = -\alpha_1(v_1)+\alpha_2(v_2)$$
where $\alpha = \pi_1^*\alpha_1+\pi_2^*\alpha_2.$
Therefore $i_1^* (-\pi_1^*\omega_1+\pi_2^*\omega_2)= \pi_1^*\omega_1 + \pi_2^*\omega_2,$
and thus $i_1$ is a symplectomorphism $T^*M_1\times T^*M_2 \to T^*M_1^- \times T^*M_2.$
$L_f$ is thus a Lagrangian submanifold of $T^*M_1^- \times T^*M_2,$ and so we have verified that $L_f$ is a canonical relation.\end{proof}

Let $f:M_1\to M_2$ and $g:M_2\to M_3$ be smooth maps.
Then
$$L_g \star L_f = \{(f^*\xi|_{x_1},\xi|_{f(x_1)},g^*\eta|_{x_2}, \eta|_{g(x_2)}) : \xi|_{f(x_1)} = g^*\eta|_{x_2}\}.$$
$\xi|_{f(x_1)} = g^*\eta|_{x_2}$ implies
$$f^*\xi|_{x_1} = f^*(\xi|_{f(x_1)}) = f^*(g^*\eta|_{x_2}) = (g\circ f)^*\eta|_{x_1}$$
and $\eta|_{g(x_2)} = \eta|_{g \circ f (x_1)}.$
Therefore,
$$\pi_{13} (L_g \star L_f) = \{ ((g\circ f)^*\eta|_{x_1}, \eta|_{g \circ f (x_1)}) \in T^*M_1^- \times T^*M_3 \}  = L_{g\circ f}.$$
But we already know $L_{g \circ f}$ is a Lagrangian submanifold of $T^*M_1^- \times T^*M_3,$ so $L_f$ and $L_g$ are composable with $L_g \circ L_f = L_{g \circ f}.$

Let $\textbf{Man}$ denote the category of manifolds.
We have now proven the following lemma which was mentioned in the introduction.

\begin{lemma} There is a functor $\textbf{Man} \to \textbf{Symp},$ which (on objects) maps a manifold (or configuration space) $M$ to its cotangent bundle $T^*M$ (with the canonical symplectic structure), and (on morphisms) maps a smooth map of manifolds (or configuration spaces) $f:M_1\to M_2$ to the canonical relation $L_f \subset T^*M_1^- \times T^*M_2$ defined by (\ref{defn3}).
\end{lemma}

In this section we have been considering finite dimensional manifolds. At least in a formal way, the results carry over
to infinite dimensional manifolds, as we will consider in section \ref{functor}.

\section{Harmonic Maps in General}

In this section we will introduce harmonic maps largely following \cite{Hel} and \cite{EeLe} but with slightly different notation.

A weakly harmonic map from a Riemannian manifold $(M^m,g)$ to a Riemannian manifold $(N^n,h)$ is an element $\phi$ of the ($L^2$) Sobolev space $W^1(M;N)$ that is a critical point of the generalized Dirichlet energy
$$\mathcal{A}(\phi) = \int_M \left<d\phi\wedge * d\phi \right>.$$
where we interpret $d\phi$ as a one form on $M$ with values in the pull back bundle $\phi^*TN \to M,$  $<\cdot,\cdot>$ is the pull back of $h$ to $\phi^*TN,$ and $*$ denotes the Hodge dual operator for forms on $(M,g).$
In coordinates on $M$
\begin{align*}
\left<d\phi\wedge * d\phi \right> & = \left<(\partial_\alpha\phi \otimes dx^\alpha)\wedge * (\partial_\beta \phi \otimes dx^\beta)\right> \\
& = \left< (\partial_\alpha\phi \otimes \partial_\beta\phi) \otimes dx^\alpha*dx^\beta\right> \\
& = \left< \partial_\alpha\phi, \partial_\beta\phi \right> g^{\alpha\beta}\text{dvol}_g.
\end{align*}
We use the usual convention by which we sum over repeated indices, and the shorthand $\partial_\alpha = \frac{\partial}{\partial x^\alpha}.$
$g^{\alpha\beta}$ is the inverse of the matrix $g_{\alpha\beta}= g(\partial_\alpha,\partial_\beta),$ and $\text{dvol}_g$ is the volume element on $M$ compatible with the metric $g.$
Moreover, if $M$ and $N$ are both open subsets of Euclidean spaces, then $\left<d\phi\wedge * d\phi \right> = \sum_\alpha |\partial_\alpha\phi|^2 dV.$
We will call a weakly harmonic map that is also smooth a (strongly) harmonic map.
It should be noted that harmonic maps are a generalization of both geodesics (when the domain is a line), and harmonic function (when the target is the real line).

A connection $\nabla$ on a vector bundle $V\to M$ is a linear map $$\Omega^0(V) \to \Omega^0(V) \otimes \Omega^1(M),$$
satisfying $\nabla(fs)=f\nabla(s)+s\otimes df$ for a function $f$ and a section $s$
i.e. $\nabla$ takes sections of $E$ to one forms with values in $E.$
We can also extend $\nabla$ to a complex of maps on differential forms with values in $E$
$$\Omega^0(\phi^*TN) \xrightarrow{d^\nabla} \Omega^0(\phi^*TN) \otimes \Omega^1(M) \xrightarrow{d^\nabla} \Omega^0(\phi^*TN) \otimes \Omega^2(M) \xrightarrow{d^\nabla} \dots$$
in a way completely analogously to the way the the differential $d$ extends to the exterior algebra on differential forms, but with the essential difference that
$(d^\nabla)^2 = R,$ the curvature tensor for $\nabla.$
This reduces to the usual exterior derivative when $E$ is the trivial flat real line bundle $\mathbb{R}\times M \to M.$

Now suppose that $\phi:M\to N$, and let $\nabla$ denote the pull back of the Levi-Civita connection on $(N,h)$ to $\phi^*TN$ (a vector bundle on $M$).
In considering variations of a map $\phi$, we will often introduce a family of maps $\phi_t$ which can be considered a map on $I\times M,$ and we will then let $\tilde\nabla$ be the pull back of the Levi-Civita connection on $(N,h)$ to $\phi_t^*TN\to I\times M.$

\begin{lemma}
Harmonic maps $\phi:M\to N$ satisfy the harmonic map equation $d^\nabla *d\phi = 0.$
\end{lemma}
\begin{proof}
Let $\delta\phi$ be a vector field on $\phi(M)$ that vanishes on $\partial M,$ and let $\phi_t$ be a deformation of $\phi$ with $\frac{d\phi}{dt}|_{t=0} = \delta\phi.$ Then
\begin{align*}
\delta\mathcal{A}(\delta\phi) & = \int_M \frac{1}{2} \left.\frac{d}{dt} \left<d\phi_t \wedge * d\phi_t \right> \right|_{t=0} \\
& = \int_M \frac{1}{2} \left.\frac{d}{dt} g^{\alpha\beta}\left<\partial_\alpha \phi_t, \partial_\beta \phi_t \right> \right|_{t=0} \text{dvol}_g \\
& = \int_M g^{\alpha\beta} \left< \tilde\nabla_{\frac{d\phi}{dt}}\partial_\alpha \phi_t|_{t=0}, \partial_\beta \phi \right> \text{dvol}_g \\
& = \int_M g^{\alpha\beta} \left< \nabla_{\partial_\alpha \phi} \delta\phi, \partial_\beta \phi \right> \text{dvol}_g.
\end{align*}
$\text{dvol}_g$ is the volume form corresponding to the metric. In coordinate $\text{dvol}_g = \sqrt{\det g} dx^1 \wedge\dots\wedge dx^m$ where $\det g$ is the determinant of $g_{\alpha\beta}.$
Now note that
\begin{align*}
d\left< \delta\phi\wedge *d\phi \right> & = d(\left< \delta\phi, \partial_\beta\phi \right>*dx^\beta) \\
& = \partial_\alpha \left(g^{\alpha\beta} \left< \delta\phi, \partial_\beta \phi \right> \sqrt{|\det g|}\right) dx^1\wedge\dots\wedge dx^m \\
& = g^{\alpha\beta} \left< \nabla_{\partial_\alpha \phi} \delta\phi, \partial_\beta \phi \right> \text{dvol}_g
+ g^{\alpha\beta}\left<\delta\phi, \nabla_{\partial_\alpha \phi}\partial_\beta \phi \right> \text{dvol}_g \dots \\
& \quad + g^{\alpha\beta} \partial_{\alpha} \left(g^{\alpha\beta} \sqrt{|\det g|}\right) \left<\delta\phi,\partial_\beta\phi \right> dx^1\wedge\dots\wedge dx^m,
\end{align*}
and
\begin{align*}
\left< \delta\phi \wedge d^\nabla *d\phi \right>  & = \left< \delta\phi \wedge d^\nabla (\partial_\beta\phi * dx^\beta)\right> \\
& = \left< \delta\phi, \nabla_{\partial_\alpha} \left(g^{\alpha\beta}\partial_\beta\phi \sqrt{|\det g|}\right)\right>dx^1 \wedge \dots \wedge dx^m \\
& = g^{\alpha\beta}\left<\delta\phi, \nabla_{\partial_\alpha \phi}\partial_\beta \phi \right> \text{dvol}_g \dots \\
& \quad + g^{\alpha\beta} \partial_{\alpha} \left(g^{\alpha\beta} \sqrt{|\det g|}\right) \left<\delta\phi,\partial_\beta\phi \right> dx^1\wedge\dots\wedge dx^m.
\end{align*}
Therefore, by using Stokes' theorem we get the integration by parts
\begin{equation}\label{intparts}\int_M g^{\alpha\beta} \left< \nabla_{\partial_\alpha \phi} \delta\phi, \partial_\beta \phi \right> \text{dvol}_g = \int_{\partial M} \left< \delta\phi\wedge *d\phi \right> - \int_M \left< \delta\phi \wedge d^\nabla *d\phi \right>.\end{equation}
Now since $\delta\phi$ vanishes on the boundary, we get
$$\delta\mathcal{A}(\delta\phi) =  - \int_M \left< \delta\phi \wedge d^\nabla *d\phi \right>.$$
$\phi$ is a harmonic map if and only if
$$\delta\mathcal{A}(\delta\phi) =  0$$
for all $\delta\phi,$ and $\phi$ is smooth.
But then $d^\nabla *d\phi = 0$ follows from our computation of $\delta\mathcal{A}.$
\end{proof}

We also have the following linearization of the harmonic map equation, which generalizes the equation for Jacobi fields along geodesics.
\begin{lemma} Suppose that $\phi$ is a harmonic map and $\phi_t$ is a deformation of $\phi$ along the space of harmonic maps.
Then $\delta\phi = \frac{d}{dt}|_{t=0}$ satisfies the linear equation
$$d^\nabla * d^\nabla \delta\phi - g^{\alpha\beta} R(\partial_\alpha \phi, \delta\phi) \partial_\beta\phi \textnormal{ dvol}_g=0$$
where $R$ is the curvature tensor for $\nabla.$
\end{lemma}

\begin{proof}
Let $X$ be a vector field on $\phi(M)$ that vanishes on $\phi(\partial M).$
Form a vector field $X_t$ on $\phi_t (M)$ by parallel transport of $X$ with respect to $\tilde\nabla.$
Then for fixed $t,$ $X_t$ vanished on $\phi_t(\partial M),$ and $\tilde\nabla_{\frac{d\phi}{dt}} X_t = 0.$
By the arguments in the previous proof $\left<X_t \wedge d^\nabla * d\phi_t\right> = 0$ almost everywhere.
Therefore the following vanishes
\begin{align*} & \frac{d}{dt} \left< X_t \wedge d^\nabla * d\phi_t \right>|_{t=0} \\ & = \left< X \wedge \tilde\nabla_{\frac{d\phi_t}{dt}} \tilde\nabla_{\partial_\alpha \phi_t} \left(g^{\alpha\beta} \partial_{\beta} \phi_t \sqrt{\det |g|} \right)|_{t=0}\right> dx^1\wedge\dots\wedge dx^m.\end{align*}
The previous holds for all $X,$ so
\begin{align*}
& \tilde\nabla_{\frac{d\phi_t}{dt}} \tilde\nabla_{\partial_\alpha \phi_t} \left(g^{\alpha\beta} \partial_{\beta} \phi_t \sqrt{\det |g|} \right)|_{t=0} dx^1\wedge\dots\wedge dx^m \\
& = \nabla_{\partial_\alpha \phi}\left(g^{\alpha\beta} \sqrt{\det |g|} \tilde\nabla_{\frac{d\phi_t}{dt}}\partial_{\beta} \phi_t|_{t=0}  \right) dx^1\wedge\dots\wedge dx^m \\ & \quad + R(\delta\phi,\partial_\alpha\phi) g^{\alpha\beta} \partial_{\beta} \phi \text{ dvol}_g \\
& = \nabla_{\partial_\alpha \phi}\left(g^{\alpha\beta}  \sqrt{\det |g|} \nabla_{\partial_{\beta} \phi}\delta\phi \right) dx^1\wedge\dots\wedge dx^m \\ & \quad - g^{\alpha\beta}R(\partial_\alpha\phi,\delta\phi) \partial_{\beta} \phi \text{ dvol}_g
\end{align*}
all vanish. Note that $\tilde\nabla_{\frac{d\phi_t}{dt}}$ ignores $g^{\alpha\beta}  \sqrt{\det |g|}$ since it is constant in t.
Expressing $d^\nabla * d^\nabla \delta\phi$ in coordinates on $M$ gives
\begin{align*}
d^\nabla * d^\nabla \delta\phi & = d^\nabla \nabla_{\partial_\beta\phi} * dx^\beta \\
& = \nabla_{\partial_\alpha \phi}\left(g^{\alpha\beta}  \sqrt{\det |g|} \nabla_{\partial_{\beta} \phi}\delta\phi \right) dx^1\wedge\dots\wedge dx^m.
\end{align*}
Therefore we get the result
$$d^\nabla * d^\nabla \delta\phi - g^{\alpha\beta} R(\partial_\alpha \phi, \delta\phi) \partial_\beta\phi \textnormal{ dvol}_g=0.$$
\end{proof}

We now recall how the functional $\mathcal{A}$ transforms under conformal transformations.
Consider a transformation of the metric $g \to e^a g$ where $a : M \to \mathbb{R}.$
Then $g^{\alpha\beta} \to e^{-a} g^{\alpha\beta}$ and $\sqrt{\det g} \to e^{\frac{ma}{2}} \sqrt{\det g},$ so
$$\mathcal{A}(\phi) = \int_M \frac{1}{2} g^{\alpha\beta}  \left< \partial_\alpha\phi, \partial_\beta\phi \right> \text{dvol}_g
\to e^{\frac{ma}{2} - a} \int_M \frac{1}{2} g^{\alpha\beta}  \left< \partial_\alpha\phi, \partial_\beta\phi \right> \text{dvol}_g.$$
In particular when $m=2$, $\mathcal{A}(\phi)$ is left unchanged by the transformation of the metric. Another way to express this is that the star operator is conformally invariant in the middle degree, which is one for a $2$-manifold,
and this is what is used in the definition of the energy functional. This implies the following

\begin{lemma}
When the domain $M$ is two dimensional, the energy functional  $\mathcal{A}$, and the harmonic map equation, are invariant with respect to locally conformal changes of the metric for $M$.
\end{lemma}

From now on, whenever we consider coordinates on a Riemann surface $\Sigma,$ we will consider characteristic coordinates $z=x+iy, \bar z = x - iy$ where $(x,y)$ are conformal coordinates.
We will also extend $<\cdot,\cdot>$ and $\nabla$ to the complexification of $\phi^*TN.$
We then have the following form of the equation for harmonic maps $\phi:\Sigma \to N$
$$\nabla_{\partial \phi} \bar\partial \phi + \nabla_{\bar\partial \phi} \partial \phi= 0,$$
and the generalization of the equation for Jacobi fields
$$\nabla_{\partial \phi}\nabla_{\bar\partial \phi} \delta\phi - R(\partial\phi, \delta\phi)\bar\partial\phi + \nabla_{\bar\partial \phi} \nabla_{\partial \phi} \delta\phi - R(\bar\partial\phi, \delta\phi) \partial\phi = 0$$
for a variation $\delta\phi$ of $\phi$  along the space of harmonic maps.

From [4] we also have the powerful result that a weak harmonic map on a Riemann surface $\Sigma$ is infinitely differentiable if $N$ is a $C^\infty$ Riemannian manifold.
In particular, any weak harmonic map satisfies the harmonic map equation.
If $N$ is only $C^k$ for some finite $k,$ then a harmonic map is as regular as the target $N.$

\section{The Functor Lemma}\label{functor}

In this section we will study the functor lemma mentioned in the introduction. In particular we fix a Riemannian target manifold $N$.

Given a compact Riemann surface $\Sigma$ with boundary $S$, our first objective is to understand the sense in which $Harm(\Sigma; N)$ is a Lagrangian submanifold of the space of Cauchy data on $S.$ To avoid technical issues, we will ultimately restrict to smooth functions on $\Sigma$ and $S$ (denoted $\Omega^0(\Sigma)$ and $\Omega^0(S)$, respectively), rather than the critical Sobolev classes  $W^1(\Sigma)$ and $W^{1/2}(S)$, respectively.

\begin{definition} Given a compact oriented $1$-manifold $S$, the space of Cauchy data is the subbundle
$$C_S:= \bigsqcup_{\phi\in \Omega^0(S;N)}\Omega^1(S;\phi^*TN) \subset T^*\Omega^0(S;N),$$
where given $\phi\in\Omega^0(S;N),$ and $\alpha \in \Omega^1(S;\phi^*TN)$
$$\alpha(X) = \int_S \left<X \wedge \alpha \right>.$$ for $X\in\Omega^0(S;\phi^*TN) = T\Omega^0(S;N) $
\end{definition}

At a heuristic level, the space of Cauchy data is essentially the same as the cotangent bundle of the configuration space $\Omega^0(S;N)$. However at a technical level there is a distinction, because when we take the dual of the infinite dimensional tangent space of $\Omega^0(S;N)$ at $\phi$, we introduce distributional type objects. We prefer to work with a more restrictive class of Cauchy data. As for the cotangent bundle, the action one form $\Theta$ on $C_S$  is given by
$$\Theta(v|_{\alpha_\phi}) = \alpha(\pi_*v|_{\alpha_\phi}) = \int_S  \left< \pi_*v|_{\alpha_\phi} \wedge \alpha_\phi \right>$$ for $v |_{\alpha_\phi} \in TC_S$, where $\pi:C_S \to \Omega^0(S;N)$ is the projection. In turn the symplectic structure on $C_S$ is given by $\omega = d\Theta$.

On $W^1(\Sigma;N)$ there is the trace map $tr : W^1(\Sigma;N) \to W^{1/2}(S;N)$ given by restriction.
However, we would like to keep track of Cauchy data (for a second order equation - the harmonic map equation), so we will introduce the enhanced trace map (and we will now restrict to smooth functions)
$$\textbf{r} : \Omega^0(\Sigma;N) \to C_S: \phi \to (*d\phi)|_S.$$ This essentially maps a field on $\Sigma$ to
its boundary values and its normal derivative along the boundary (which is a coordinate free version of its velocity
vector).

Following \cite{Cat1}, let $\delta$ be the exterior derivative on $\Omega^0(\Sigma; N).$ From the previous section (see
(\ref{intparts})),
where we now keep track of boundary terms, for $\phi:\Sigma\to N$ and $X=\delta\phi \in \Omega^0(\Sigma;\phi^*TN)$ (a tangent vector at $\phi$)
$$\delta\mathcal{A}(X|_{\phi}) = - \int_\Sigma \left< X|_\phi \wedge d^\nabla *d\phi \right> +\int_{S} \left< X|_{\phi} \wedge *d\phi|_S \right>.$$
When $\phi$ is a harmonic map, the integral over $\Sigma$ vanishes, and the integral over $S$ is just $\textbf{r}^*\Theta.$
Therefore,
$$\delta\mathcal{A}|_{Harm(\Sigma;N)} = \textbf{r}^* \alpha|_{Harm(\Sigma;N)}.$$
Upon applying $\delta$ to both sides of the equality, we get
$$\delta\textbf{r}^* \alpha|_{Harm(\Sigma;N)} = \omega|_{Harm(\Sigma; N)}=0,$$
since exterior differentiation commutes with pull back.
Therefore, at all smooth points, $\mathbf{r}(Harm(\Sigma; N))$ is an isotropic submanifold of the symplectic manifold $C_S.$

\begin{lemma} At all smooth points, $\mathbf{r}(Harm(\Sigma; N))$ is a Lagrangian submanifold of $C_S$.
\end{lemma}

\begin{proof} We have already established that $\mathbf{r}(Harm(\Sigma; N))$ is isotropic. To prove maximality consider the diagram
$$ \begin{matrix} Harm(\Sigma;N) &\stackrel{\mathbf{r}}{\rightarrow}& C_S\\
& \searrow & \downarrow \pi\\ & & \Omega^0(S;N) \end{matrix} $$
Given $\phi\in Harm(\Sigma;N)$, there is a corresponding diagram of tangent spaces and derivatives. Because
the tangent space for $Harm(\Sigma;N)$ is defined by the Jacobi equation, an elliptic differential equation,
the derivative of the composed map $Harm(\Sigma;N)\to \Omega^0(S;N)$ is surjective, i.e. we can always solve the
linear Dirichlet problem. This implies that at a smooth point, the tangent space to
$\mathbf{r}(Harm(\Sigma; N))$ is maximal isotropic in $C_S$ [In general, given a map $X\to T^*C$ with isotropic image,
if the composed map $X\to C$ has surjective derivative, then the image of $X$ is Lagrangian at smooth points].
\end{proof}

To make sense of the functor lemma, we also need to introduce Segal's category of Riemann surfaces which we will denote as $\textbf{Riem}.$

\begin{definition}(Segal's Category of Riemann Surfaces)
The objects in $\textbf{Riem}$ are compact oriented $1$-manifolds, and a morphism $S^-\to S^+$ is a compact Riemann surface $\Sigma,$ (with its standard orientation) such that (the boundary) $\partial\Sigma = S^+-S^-$, where this notation means that $S^+$ is the portion of the boundary where the intrinsic orientation and the induced orientations agree. Given compact Riemann surfaces $\Sigma_1$ and $\Sigma_2$ with $\partial\Sigma_1 = S_1^+ -S_1^-$ and $\partial\Sigma_2 = S_2^+- S_2^-$, where $S_1^+=S_2^-$, the composition $\Sigma_2\circ\Sigma_1$ is given by gluing the Riemann surfaces along the common boundary $S_1^+=S_2^-$.
\end{definition}

Suppose that $\Sigma:S^- \to S^+$ is a morphism in the category $\textbf{Riem}.$
The reversal of the orientation of $S_-$ changes the sign of the symplectic form on the restriction of $C_S$ to $S^-,$ so $C_{S^-}^- \times C_{S^+} = C_S$ as infinite dimensional symplectic manifolds.
Therefore, $Harm(\Sigma;N)$ is a canonical relation from $C_{S^-}$ to $C_{S^+}.$
The identification $\Sigma \to Harm(\Sigma;N)$ is the functor from $\textbf{Riem}$ to $\textbf{Symp}.$

It is easy to see that the above identification is in fact functorial as follows.
Let $\Sigma_1$ and $\Sigma_2$ be two elements of $\textbf{Riem}$ such that $S_1^+$ is identified with $S_2^-.$
Then  $Harm(\Sigma_2;N) \circ Harm(\Sigma_1;N)$ is the subset of $C_{S_1^-}^- \times C_{S_2^+}$ consisting of elements of the form $(*\phi_1|_{S_1^-},*\phi_2|_{S_2^+}),$  where $\phi_1:\Sigma_1 \to N$ and $\phi_2:\Sigma_2 \to N$ are harmonic maps such that $*d\phi_1|_{S_1^+} = *d\phi_2|_{S_2^-}.$
Suppose $\phi$ is a harmonic map on $\Sigma,$ then $\phi_1 = \phi|_{\Sigma_1}$ and $\phi_2 = \phi|_{\Sigma_2}$ are harmonic maps on $\Sigma_1$ and $\Sigma_2$ respectively such that $*d\phi_1|_{S_1^+} = *d\phi_2|_{S_2^-},$ so $\mathcal{H}(\Sigma;N) \subset Harm(\Sigma_2;N) \circ Harm(\Sigma_1;N).$
Now suppose $\phi_1:\Sigma_1 \to N$ and $\phi_2:\Sigma_2 \to N$ are harmonic maps such that $*d\phi_1|_{S_1^+} = *d\phi_2|_{S_2^-}.$
Define a map $\phi : \Sigma \to N$ by $\phi|_{\Sigma_1} = \phi_1$ and $\phi|_{\Sigma_2} = \phi_2.$ Since the values
and normal derivatives match up along the boundary, $\phi$ is a harmonic map on $\Sigma.$

This completes the proof of the Functor Lemma \ref{sewing}.

\end{document}